\documentclass[a4paper,11pt,reqno]{amsart}

\usepackage{bbm}

\usepackage[latin1]{inputenc}
\usepackage[T1]{fontenc}
\usepackage[english]{babel}
\usepackage{latexsym,amsmath,amssymb,amsfonts}
\usepackage{dsfont}
\usepackage{color}
\usepackage{graphicx}
\usepackage{float}
\usepackage{esint}

\theoremstyle{plain}
\begingroup
\newtheorem{theorem}{Theorem}[section]
\newtheorem{lemma}[theorem]{Lemma}
\newtheorem{proposition}[theorem]{Proposition}

\endgroup
\theoremstyle{definition}
\begingroup
\newtheorem{definition}[theorem]{Definition}
\newtheorem{remark}[theorem]{Remark}

\endgroup
\theoremstyle{remark}

\numberwithin{equation}{section}

\newcommand{\N}{\mathbb N}

\newcommand{\R}{\mathbb R}

\newcommand{\hno}{{\mathcal H}^{N-1}}

\newcommand{\dd}{\mathrm{d}}

\newcommand{\dhno}{\;\dd{\mathcal H}^{N-1}}

\newcommand{\dx}{\;\dd x}
\newcommand{\dy}{\;\dd y}

\newcommand{\meno}{\!\setminus\!}

\renewcommand{\o}{\Omega}
\renewcommand{\sp}{BV(\o;\R^M)}
\newcommand{\rb}{\partial^*}

\newcommand{\f}{\mathcal{F}}
\newcommand{\g}{\mathcal{G}}
\newcommand{\per}{\mathcal{P}}

\newcommand{\restr}{%
  \,\raisebox{-.127ex}{\reflectbox{\rotatebox[origin=br]{-90}{$\lnot$}}}\,%
}

\newcommand{\average}{{\mathchoice {\kern1ex\vcenter{\hrule height.4pt
width 6pt
depth0pt} \kern-9.7pt} {\kern1ex\vcenter{\hrule height.4pt width 4.3pt
depth0pt}
\kern-7pt} {} {} }}

\allowdisplaybreaks

\begin{document}

\title{Piecewise constant reconstruction of damaged color images}
\author{Riccardo Cristoferi, Irene Fonseca}
\address{Center for Nonlinear Analysis, Carnegie Mellon University, Pittsburgh, PA 15213-3890, USA}
\subjclass[2010]{49J99, 26B30, 68U10.}
\keywords{Energy minimization, RGB total variation models, colorization, inpainting, image
restoration.}

\begin{abstract}
A variational model for reconstruction of damaged color images is studied, in particular
in the case where only finitely many colors are admissible for the reconstructed image.
An existence result and regularity properties of minimizers are presented.
\end{abstract}

\maketitle

\section{Introduction}

The aim of this paper is to study a variational model for the reconstruction of color images when information on the color is available everywhere except in a damaged region, where only a grey level function is known.

The variational approach we consider here has been introduced by Fornasier in \cite{For}, as part of a project aimed at restoring the Mantegna's fresco in the Ovetari Chapel of Italian Eremitani's Church in Padua.
The model is inspired by the famous ROF model for denoising, introduced by Rudin, Osher and Fatemi in the context of grey level functions (see \cite{ROF}): to minimize
\begin{equation}\label{eq:rof}
v\in BV(\o)\mapsto |Dv|(\o)+\lambda\|v-v_0\|^2_{L^2(\o)}\,,
\end{equation}
where $\o\subset\R^2$ denotes the image domain, $v_0\in L^2(\o)$ is the given image, and $\lambda\geq0$ is a tunning parameter.
In order to be able to reconstruct edges in the image, the space of functions of bounded variations $BV$ is typically used for representing an image.

When dealing with color images, there are two preferred ways to represent them mathematically.
The first one is the RGB (red, green, blue) model, where an image is represented via its three channels $u_R, u_G$ and $u_B$, with $u\in BV(\o;\R^3)$ defined as $u=(u_R,u_G,u_B)$.
The other way to represent an image is called Chromaticity/Brightness, where a RGB image $u\in BV(\o\,;\,\R^3)$ is decomposed into two components: its chromaticity $u/|u|$ and its brightness $|u|$.
The main idea of this model is to reconstruct the two parts independently
(see, for instance, \cite{FerFonMAs} and \cite{KanMar}).

The total variation model introduced by Fornasier in \cite{For} is a variant of \eqref{eq:rof}, and it appeals to the RGB model to represent the image.
The grey level information in the damage region $D\subset\o$ is modeled as a nonlinear distortion of the colors, $\mathcal{L}:\R^3\to\R$. Often $\mathcal{L}$ is taken to be of the form
\[
\mathcal{L}(v):=L(v\cdot e)\,,
\]
where $e\in\R^3$ is a unit vector and $L:\R\to[0,\infty)$ is an increasing function (usually neither concave nor convex).
Usually, both $e$ and $L$ are chosen based on the given image in order to best fit (\emph{i.e.}, with minimal total variance) the distribution of data from the real color (see \cite{ForMar}).
The functional to be minimized is
\[
\f(u):=|Du|(\o)+\lambda\int_{\o\,\meno \,D}|u-f|^p\dx + \mu\int_D \bigl|\, L(u\cdot e)-L(f\cdot e)  \,\bigr|^p\dx\,,
\]
where $p\geq1$, $f\in L^p(\o\,;\,\R^M)$ is the given image, and $\lambda,\mu\geq0$ are tuning parameters.
Since the restored image will be colored everywhere, the problem we consider can be seen as a generalization of inpainting (see \cite{AmbMas, BalBerCasSapVer, BugBerCas, ChaKanHaShe, ChaShe1, ChaShe2}).
Note that here, for the sake of mathematical abstraction, we consider the target space to be $\R^M$, for $M\in\N$, in place of $\R^3$ as in the description of the RGB model.

We note that in the literature there are other approaches to the reconstruction of an image when information of the colors is not everywhere available (see, for instance, \cite{Irony, LevLisWei, Sap, SapYat, Sykora}).\\

Numerical experiments, as well as a first study of the model, are present in the work of Fornasier and March (see
\cite{ForMar}).
Subsequently, a rigorous analytical study in the case of perfect reconstruction (\emph{i.e.}, when $\lambda=\mu=\infty$) has been carried out by Fonseca, Leoni, Maggi and Morini in \cite{FLMM}. In particular, the authors provide a characterization of the piecewise constant functions $f$ that can be obtained as minimizers of $\f$, whose jump set is the union of finitely many Lipschitz curves.
Furthermore, in \cite{FLMM} the authors study the minimizers in the case in which the damaged region is uniformly distributed in $\o$.\\

In this paper we pursue the study initiated in \cite{FLMM}, working in the general case where $u\in BV(\o\,;\,\R^M)$, with $\o\subset\R^N$ a bounded connected open set with Lipschitz boundary, and where we fix apriori the number of colors that we are allowed to use, say $k\in\N$, but the color spectrum is not restricted, \emph{i.e.}, we consider the minimization problem
\begin{equation}\label{eq:kcolors}
\min\left\{\, \f(u) \,:\, u\in BV(\o\,;\, A)\,,\, A\subset\R^M \text{ with } \mathcal{H}^0(A)=k \,\right\}\,.
\end{equation}
Notice that if $A=\{a_1,\dots,a_k\}\subset\R^M$ it is possible to write a function $u\in BV(\o\,;\, A)$ as
\begin{equation}\label{eq:writingu}
u=\sum_{i=1}^k a_i\chi_{\o_i}\,,
\end{equation}
where $\o_i:=\{ x\in\o\,:\,u(x)=a_i \}$, and the functional $\f$ becomes
\begin{align}\label{eq:functf}
\f(u)&=\sum_{i=2}^k \sum_{j<i} |a_i-a_j|\hno(\rb\o_i\cap\rb\o_j\cap\o) + \lambda\sum_{i=1}^k \int_{\o_i\,\meno\, D} |a_i-f|^p\dx \nonumber\\
&\hspace{1cm}+ \mu\sum_{i=1}^k \int_{\o_i\cap D} |L(a_i\cdot e)-L(f\cdot e)|^p\dx\,,
\end{align}
where $\rb\o_i$ denotes the reduced boundary of the set $\o_i$ (see Definition \ref{def:rb}), that coincides with the topological boundary  $\partial\o_i$ in the case it is Lipschitz, and $\hno$ is the $(N-1)$-dimensional Hausdorff measure (see Definition \ref{def:h}).
The minimization problem \eqref{eq:kcolors} can be thought both as a combination of an inpainting and a segmentation problem,
or as a partition problem with weighted perimeter and volume terms.

A popular way to segment an image is by using the Mumford-Shah functional, introduced in \cite{MS2}.
The functional is defined over couples $(v,\Gamma)$, where $\Gamma\subset\o$ is a closed set, and $v\in C^1(\o\setminus\Gamma)$, and reads as
\begin{equation}\label{eq:ms}
\mathcal{MS}(v,\Gamma):= \int_{\o\setminus \Gamma} |\nabla v|^2 \dx + \alpha\int_\o |v-f|^2\dx + \beta\hno(\Gamma)\,,
\end{equation}
where $\alpha,\beta>0$ are tuning parameters. The set $\Gamma$ represents the set of edges of the objects in the image, and this is assumed to be closed.
Existence for the minimization problem
\begin{equation}\label{eq:minms}
\min\{\, \mathcal{MS}(v,\Gamma) \,:\, \Gamma\subset\o \text{ is a closed set}\,,\, v\in C^1(\o\setminus \Gamma) \,\}\,,
\end{equation}
has been proved by De Giorgi, Carriero e Leaci in \cite{DGCL} via the following relaxed version of the functional \eqref{eq:ms},
\begin{equation}\label{eq:relms}
\widetilde{\mathcal{MS}}(w):= \int_{\o\setminus J_w} |\nabla w|^2 \dx + \alpha\int_\o |w-f|^2\dx + \beta\hno(J_w)\,,
\end{equation}
where $w\in SBV(\o)$ (the space of special functions with bounded variation, see \cite[Chapter 4]{AFP})
 and $J_w$ denotes the jump set of $w$ (see Definition \ref{def:jump}). Existence of a solution to the minimization problem
\begin{equation}\label{eq:minrelms}
\min\{\, \widetilde{\mathcal{MS}}(w) \,:\, w\in SBV(\o) \,\}
\end{equation}
can be obtained via the Direct Method of the Calculus of Variations. In order to get a solution to the minimization problem \eqref{eq:minms}, a solution $w$ to the minimization problem \eqref{eq:minrelms} is a solution of the original problem \eqref{eq:minms} provided the jump set $J_w$ of $w$ is \emph{essentially closed}, namely that $\hno\left((\overline{J_w}\setminus J_w)\cap\o\right)=0$, in which case we set $\Gamma:=\overline{J_w}$ and $v:=w$.
The proof of the fact that the jump set is essentially closed relies on delicate density estimates for the jump set of a solution $w$ to the minimization problem \eqref{eq:minrelms}. Once that property is establish, regularity theory allows to conclude that $w\in C^1(\o\meno\overline{J_w})$.

Subsequently, Congedo and Tamanini proved existence and regularity properties of minimizers to the minimization problem \eqref{eq:ms}
in the case where $\nabla v=0$ in $\o\meno\Gamma$ (see \cite{ConTam}, \cite{TamCon}):
\begin{equation}\label{eq:minpbCT}
\min\{\, \mathcal{MS}(v,\Gamma) \,:\, \Gamma\subset\o \text{ is a closed set}\,,\,\nabla  v=0 \text{ in } \o\setminus K \,\}\,.
\end{equation}
Inspired by \cite{DGCL}, the main idea in \cite{TamCon} is to rephrase the problem in the space of Caccioppoli partitions (see Definition \ref{def:Cacc}), namely to consider the functional
\begin{equation}\label{eq:ConTam}
G(\mathcal{W}, w):=\sum_{i=2}^\infty\sum_{j<i} \hno(\rb W_i\cap\rb W_j\cap\o) + \alpha\sum_{i=1}^\infty \int_{W_i} |w_i-f|^p\dx
\end{equation}
where $\mathcal{W}:=\{W_i\}_{i\in\N}$ is a partition of $\o$ (see Definition \ref{def:Cacc}), and $w\in SBV(\o)$ is given by
$w:=\sum_{i\in\N}w_i\chi_{W_i}$, for some $w_i\in\R$.
Notice that here countably many partitions are allowed.
If $(\mathcal{W}, w)$ is a solution to the minimization problem
\begin{equation}\label{eq:minpbCT}
\min\{ G(\mathcal{W}, w) : \mathcal{W} \text{ partition of } \o,\, w\in SBV(\o)\,, \nabla w=0 \text{ on } \o_i, \text{ for all } i\in\N \}\,,
\end{equation}
and if
\begin{equation}\label{eq:clousure}
\hno\left((\overline{J_w}\setminus J_w)\cap\o\right)=0
\end{equation}
then $w$ is also a solution to the minimization problem \eqref{eq:minpbCT}.
In order to obtain \eqref{eq:clousure}, the argument is anchored to an elimination lemma proved by Congedo and Tamanini (see \cite[Lemma 5.3]{TamCon}).

\begin{lemma}\label{lem:lemCT}
Let $p\geq1$ and let $(\mathcal{W}, w)$ be a solution to the minimization problem \eqref{eq:minpbCT}, and assume that $f\in L^{Np}_{loc}(\o)$.
Then for all $m\in\N$, there exists $\eta>0$ such that for all $x\in\o$ there exists $r_0>0$ with the following property:
if
\[
\Bigl|B_r(x)\setminus \bigcup_{i=1}^m W_i\Bigr|\leq\eta r^N
\]
for some $0<r<r_0$, then $|B_{r/2}(x)\setminus \bigcup_{i=1}^m W_i|=0$.
\end{lemma}

The importance of Lemma \ref{lem:lemCT} relies on the fact that it allows to locally reduce the complexity of the partition $\mathcal{W}$. Indeed, using Lemma \ref{lem:lemCT} one can prove that for $\hno$-a.e. $x\in J_w$, there exists a ball $B_r(x)$ such that $B_r(x)=W_i\cap W_j\cap J_w$, for some $i,j\in\N$. Thus, in $B_r(x)$ the partition problem becomes a problem of least area with a volume term, leading to \eqref{eq:clousure}, as well as to regularity properties of the interfaces $\rb W_i$.
In the case where the measure of the interfaces is weighted (as in the first term of \eqref{eq:functf}), a similar result as Lemma \eqref{lem:lemCT} holds under the additional assumption that $|a_i-a_j|<|a_i-a_k|+|a_k-a_j|$, whenever the indexes $i,j,k$ are different (see \cite{Leo}).\\

The main result of the paper is the following existence and regularity result for a solution to the minimization problem \eqref{eq:kcolors}.
\begin{theorem}\label{thm:kcolors}
Let $\o\subset\R^N$ be an open connected set with Lipschitz boundary. Let $p\geq1$ and $f\in L^p(\o;\R^M)$ be such that $L(f\cdot e)\in L^p(\o)$.
Then the minimization problem \eqref{eq:kcolors} admits a solution $u\in BV(\o\,;A)$, where $A:=\{ a_1,\dots,a_k\}\subset\R^M$, \emph{i.e.},
\[
u=\sum_{i=1}^k a_i\chi_{\o_i}\,,
\]
with $\o_i=\o_i(1)$ (the points of density $1$ for $\o_i$, see Definition \ref{def:ptdensity}), for every $i=1,\dots,k$.

Assume, in addition, that
\begin{itemize}
\item[(H1)] $f\in L^{q}(\o;\R^M)$,
\item[(H2)] $L(f\cdot e)\in L^{q}(\o;\R^M)$,
\end{itemize}
for some $q\geq N(p-1)$. If a solution $u$ is such that $|a_i-a_j|<|a_i-a_k|+|a_k-a_j|$ whenever the indexes $i,j,k$ are different, then the following hold:
\begin{itemize}
\item[(i)] $\o_i$ is open, for every $i=1,\dots,k$, and thus $J_u=\overline{J_u}\cap\o$,
\item[(ii)] each $\partial^*\Omega_i$ is the union of relatively open sets of class $C^{1,\alpha}$, where $\alpha:=\frac{1}{2}\left(1-\frac{N(p-1)}{q}\right)$ and a closed singular set of $\hno$ measure zero,
\item[(iii)] there exists $\beta>0$, depending on $\o, D, N, \lambda, \mu, \|f\|_{L^{q}}, \|L(f\cdot e)\|_{L^{q}}$, such that
\[
\liminf_{\rho\to 0}\frac{\hno(J_u\cap B_\rho(x))}{\rho^{N-1}}\geq\beta\,,
\]
for every $x\in \overline{J_u}\cap\o$.
\end{itemize}
\end{theorem}

\begin{remark}
The additional regularity assumptions on $f$ and on $L(f\cdot e)$ that we require in Theorem \ref{thm:kcolors}, $(H1)$ and $(H2)$ respectively, are similar in spirit to the ones required by Congedo and Tamanini in \cite{TamCon} for partial regularity of the interfaces for the minimization problem \eqref{eq:minpbCT}. In their case they have to require $f\in L^{Np}(\o)$, and they provide a counterexample to the regularity of interfaces if $f$ is less integrable. Here we need lower integrability of $f$ and of $L(f\cdot e)$ due to the fact that our functional has a weighted perimeter term rather than just the perimeter itself.
\end{remark}

In order to get the regularity properties claimed in Theorem \ref{eq:kcolors}, we first consider the case in which the set $A=\{a_1,\dots,a_k\}\subset\R^M$ is fixed apriori. By rephrasing the problem in the space of Caccioppoli partitions of $\o$, we obtain in Theorem \ref{thm:kfixedcolors} the existence of solutions for the minimization problem
\[
\min\left\{\, \f(u) \,:\, u\in BV(\o\,;\, A) \,\right\}\,,
\]
by using the lower semi-continuity result for functionals defined on partitions due to Ambrosio and Braides (see \cite{AmbBra2}).
This framing of the problem allows us to prove regularity properties of a minimizer (see Theorem \ref{thm:kfixedcolors}).
The main technical result of this paper is Theorem \ref{thm:elimination}, that extends to \eqref{eq:functf} the elimination lemma proved by Leonardi
in the case $\lambda=\mu=0$ (see \cite[Theorem 3.1]{Leo}).

We also present two results of independent interest: in Theorem \ref{thm:exi} we prove the existence of a solution to the minimization problem
\[
\min\left\{\, \f(u) \,:\, u\in BV(\o\,;\, \R^M) \,\right\}\,,
\]
thus extending the one of Fornasier and March (see Theorem \ref{thm:exi}) since we don't assume any apriori bound on the given image $f$.
Finally, in Proposition \ref{prop:nontri} we characterize the functions $L:\R\to[0,\infty)$ for which the functional $\f$ is non trivial.\\

Further regularity properties of minimizers for the functional \eqref{eq:functf} in the case where no restrictions on the class of minimizers, as well as in the particular case of finitely many admissible colors for the reconstruction, are currently under investigation. Also a characterization of the piecewise constant functions $f$ that can be obtained as minimizers of $\f$, whose jump set is the union of finitely many Lipschitz curves in the spirit of the result of \cite{FLMM}, is being undertaken.
Finally, a study of the model where the total variation is replaced by an anisotropic functional of the total variation of $u$ will be carried out in a future work.\\

The paper is organized as follows.
After recalling basic notions and background in Section \ref{sec:prel}, an existence result for the minimization problem is presented in Section \ref{sec:ex}, while non triviality is studied in Section \ref{sec:nontri}.
Finally, Section \ref{sec:pc} is devoted to the study of the existence and regularity properties of minimizers in the case in which only finitely many colors are allowed for the reconstructed image.


\section{Preliminaries}\label{sec:prel}

In this section we recall some basic notions on BV functions and sets of finite perimeter. For a reference see, for instance, \cite{AFP}.

\begin{definition}
Let $A\subset\R^N$ be an open set. A function $u\in L^1(A\,;\,\R^M)$ is said to be of \emph{bounded variation} if
\[
|Du|(A):=\sup\left\{\, \int_A u\cdot\mathrm{div}\,\varphi\dx
    : \varphi\in \left[C^\infty_c(A\,;\,\R^N)\right]^M\,, |\varphi|_{L^\infty}\leq1
      \,\right\}<\infty\,.
\]
We write $u\in BV(A\,;\,\R^M)$.
\end{definition}

\begin{remark}
For a Borel set $B\subset A$, the function $B\to |Du|(B)$ is a finite Radon measure that is lower semi-continuous with respect to the $L^1$ convergence of sets.
\end{remark}

\begin{definition}\label{def:jump}
Let $u\in BV(\o;\R^M)$ and $x\in\o$.
We say that $x$ is an \emph{approximate jump point} of $u$ if there exists $a,b\in\R^M$, with $a\neq b$, and $\nu\in \mathbb{S}^{N-1}$ such that
\begin{align*}
&\lim_{r\to0}\frac{1}{|B^+_r(x,\nu)|} \int_{B^+_r(x,\nu)}|u(y)-a| \dy=0\,,\\
&\lim_{r\to0}\frac{1}{|B^-_r(x,\nu)|} \int_{B^-_r(x,\nu)}|u(y)-b| \dy=0\,,
\end{align*}
where $B^\pm_r(x,\nu):=\{ y\in \R^N \,:\, \langle y-x,\pm \nu \rangle >0 \}$.
We denote by $J_u$, the \emph{jump set} of $u$, the set of points of $\o$ where this property does not hold.
\end{definition}

A special case of functions of bounded variation are those that are characteristic functions of sets of finite perimeter.

\begin{definition}
Let $A\subset\R^N$ be a Borel set.
A measurable set $E\subset\R^N$, with $|E\cap A|<\infty$, is said to have \emph{finite perimeter} in $A$ if $\chi_E\in BV(A)$. In this case, we denote by $\mathcal{P}(E\,;A)$ the total variation of $\chi_E$ in $A$, $|D\chi_E|(A)$.
\end{definition}

In order to state the structure theorem for sets of finite perimeter, we first need some definitions.

\begin{definition}\label{def:ptdensity}
Let $E\subset\R^N$ be a Borel set, and let $t\in[0,1]$.
We say that $x\in\R^N$ is a \emph{point of density $t$ for $E$} if
\[
\lim_{r\to0}\frac{|E\cap B_r(x)|}{\omega_N r^N}=t\,.
\]
The set of points of density $t$ of $E$ will be denoted by $E(t)$.
\end{definition}

\begin{definition}\label{def:rb}
Let $E\subset\R^N$ be a set of finite perimeter in some open set $A\subset\R^N$. We define the \emph{reduced boundary} $\rb E$ of $E$ as the set of points $x\in\mathrm{supp}|D\chi_E|\cap A$ such that the limit
\[
\nu_E(x):=\lim_{r\to0}\frac{D\chi_E(B_r(x))}{|D\chi_E|(B_r(x))}
\]
exists and $|\nu_E(x)|=1$.
Here $\mathrm{supp}|D\chi_E|$ denotes the support of the measure $|D\chi_E|$.
We call $\nu_E(x)$ the \emph{generalized inner normal} to $E$ at $x$.
\end{definition}

\begin{definition}\label{def:h}
Let $k\in\N$. With $\mathcal{H}^k(E)$ we denote the $k^{th}$-Hausdorff measure of a set $E\subset\R^N$.
\end{definition}

\begin{definition}
Let $E\subset\R^N$ be an $\mathcal{H}^k$-measurable set. We say that $E$ is
\emph{countably $\mathcal{H}^k$-rectifiable}, if $\mathcal{H}^k(E)<\infty$ and there exist $\{f_n\}_{n\in\N}$, with $f_n:\R^k\to\R^N$ a Lipschitz function for all $n\in\N$, such that
\[
\mathcal{H}^k \left(\, E\setminus\bigcup_{n\in\N}f_n(\R^k)\,\right)=0\,.
\]
\end{definition}

The following structure theorem for sets of finite perimeter is due to De Giorgi (see, for instance, \cite[Theorem 3.59]{AFP})

\begin{theorem}
Let $E\subset\R^N$ be a set of finite perimeter in an open set $A\subset\R^N$.
Then $\rb E\cap A$ is $\hno$-rectifiable and $|D\chi_{E\cap A}|=\hno\restr\left(\rb E\cap A\right)$.
\end{theorem}


\section{Setting of the Problem}\label{sec:set}

Let $\o\subset\R^N$ be a bounded connected open set with Lipschitz boundary, and let $D\subset\o$ be a Borel set with non empty interior and such that $|\o\meno D|>0$.
For $M\in\N$ fixed, define the functional $\f:\sp\rightarrow[0,+\infty]$ by
\begin{equation}\label{eq:f}
\f(u):=|Du|(\o)+\lambda\int_{\o\,\meno \,D}|u-f|^p\dx + \mu\int_D \bigl|\, L(u\cdot e)-L(f\cdot e)  \,\bigr|^p\dx\,,
\end{equation}
where $\lambda, \mu>0$, $p\in[1,+\infty)$ and $e\in\R^M$, with $|e|=1$, are the parameters of the model.
The function $f\in L^p(\o;\R^M)$ is the given image, and $L:\R\rightarrow[0,+\infty)$ is a continuous function.
Here $u\mapsto L(u\cdot e)$ plays the role of a nonlinear distortion.
We consider the minimization problem
\begin{equation}\label{eq:minpb}
\min_{u\in BV(\o;\R^M)} \f(u)\,.
\end{equation}
When relevant, we will stress the dependence of $\f$ and of the above minimization problem on the initial data $f$, by referring to it as the functional and minimization problem relative to $f$.


\section{Existence of a solution for the minimization problem}\label{sec:ex}

This section is devoted to showing that the minimization problem \eqref{eq:minpb} admits a solution. 
The proof relies on the following Poincar\'{e} type of inequality.

\begin{lemma}\label{lem:Poi}
There exists $C=C(\o,D)>0$ such that for all
$v\in \sp$ it holds
\[
\|v\|_{\sp}\leq C\left[\, |Dv|(\o) + \|v\|_{L^1(\o\,\meno\, D;\R^M)}  \,\right]\,.
\]
\end{lemma}

\begin{proof}
By arguing component by component, it suffices to prove the result in the case $M=1$.

Assume that the statement of the lemma does not hold. Then, there would exists a sequence $\{v_n\}_{n\in\N}\subset\sp$
such that
\begin{equation}\label{eq:absineq}
\|v_n\|_{\sp}> n\left[\, |Dv_n|(\o) + \|v_n\|_{L^1(\o\,\meno\, D)}  \,\right]\,.
\end{equation}
In particular, $\|v_n\|_{L^1(D)}>0$.
Set
\[
\bar{v}_n:=\frac{v_n}{\|v_n\|_{L^1(D)}}\,.
\]
Then \eqref{eq:absineq} becomes
\[
|D\bar{v}_n|(\o)+1+\|\bar{v}_n\|_{L^1(\o\,\setminus\, D)}> n\left[\, |D\bar{v}_n|(\o) + \|\bar{v}_n\|_{L^1(\o\,\meno\, D)}  \,\right]\,,
\]
and so
\begin{equation}\label{eq:absineq1}
|D\bar{v}_n|(\o)+\|\bar{v}_n\|_{L^1(\o\,\meno\, D)}<\frac{1}{n-1}\,,
\end{equation}
and since $\|\bar{v}_n\|_{L^1(D)}=1$, we obtain
\[
|D\bar{v}_n|(\o)+\|\bar{v}_n\|_{L^1(\o)}<1+\frac{1}{n-1}\,.
\]
Hence, up to a (not relabeled) subsequence, it holds that $\bar{v}_n\rightarrow \bar{v}$ in $L^1(\o)$ for some $\bar{v}\in\sp$.
By \eqref{eq:absineq1}, we have
\[
|D\bar{v}|(\o)=0\,,\quad\quad\quad \|\bar{v}\|_{L^1(\o\,\meno\, D)}=0\,,
\]
thus $\bar{v}\equiv0$. But this is in contradiction with the fact that $\|\bar{v}_n\|_{L^1(D)}\equiv 1$ for all $n\in\N$ implies $\|\bar{v}\|_{L^1(D)}=1$.
\end{proof}

\begin{theorem}\label{thm:exi}
The minimization problem \eqref{eq:minpb} admits a solution.
\end{theorem}

\begin{proof}
Let $\{u_n\}_{n\in\N}$ be a minimizing sequence. Without loss of generality, we can assume that
\[
\Lambda:=\lim_{n\rightarrow+\infty}\f(u_n)<+\infty\,.
\]
Then
\begin{equation}\label{eq:est1}
\sup_{n\in\N}|Du_n|(\o)<+\infty\,,
\end{equation}
and
\begin{align}\label{eq:est2}
\int_{\o\,\meno\, D} |u_n|\dx &\leq \int_{\o\,\meno\, D}|f|\dx+\int_{\o\,\meno\, D}|u_n-f|\dx \nonumber\\
&\leq |\,\o\meno\, D|^{\frac{1}{p'}}\Biggl[\,  \Bigl(\,\int_{\o\,\meno\, D}|f|^p\dx \,\Bigr)^{\frac{1}{p}}+
    \Bigl(\,\int_{\o\,\meno\, D}|u_n-f|^p \,\Bigr)^{\frac{1}{p}} \,\Biggr] \nonumber\\
&\leq |\o\,\meno\, D|^{\frac{1}{p'}} \left[\, \|f\|_{L^p(\o)}+\left(\frac{\Lambda}{\lambda}\right)^{\frac{1}{p}} \,\right]\,.
\end{align}
Applying Lemma \ref{lem:Poi} and using \eqref{eq:est1} and \eqref{eq:est2}, we get
\[
\sup_{n\in\N} \|u_n\|_{BV(\Omega;\R^M)}<+\infty\,,
\]
and so there exists $u\in \sp$ such that, up to a (not relabeled) subsequence, $u_n\rightarrow u$ in $L^1(\o;\R^M)$.
The lower semicontinuity of the total variation, together with \eqref{eq:est1}, yields
\begin{equation}\label{eq:lsctv}
|Du|(\o)\leq\liminf_{n\rightarrow+\infty}|Du_n|(\o)<+\infty\,,
\end{equation}
so that $u\in \sp$.
Up to extracting a further (not relabeled) subsequence, we can also assume that $u_n\rightarrow u$ pointwise a.e. in $\o$.
Using Fatou's lemma, we get
\begin{equation}\label{eq:lscg1}
\int_{\o\,\meno\, D}|u-f|^p\dx\leq\liminf_{n\rightarrow+\infty}\int_{\o\,\meno\, D}|u_n-f|^p\dx\,,
\end{equation}
and, recalling that the continuity of $L$ yields that $L(u_n\cdot e)\rightarrow L(u\cdot e)$ pointwise a.e. in $D$,
\begin{equation}\label{eq:lscg2}
\int_D \bigl|\, L(u\cdot e)-L(f\cdot e)  \,\bigr|^p\dx
    \leq\liminf_{n\rightarrow+\infty} 
    \int_D \bigl|\, L(u_n\cdot e)-L(f\cdot e)  \,\bigr|^p\dx\,.
\end{equation}
Hence, by \eqref{eq:lsctv}, \eqref{eq:lscg1} and \eqref{eq:lscg2}, we obtain
\[
\f(u)\leq \liminf_{n\rightarrow+\infty} \f(u_n)=\inf_{BV(\o;\R^M)} \f\,,
\]
so that $u$ is a solution of the minimization problem \eqref{eq:minpb}.
\end{proof}

\begin{remark}
The above existence theorem extends the one obtained in \cite{ForMar}, since we do not assume  apriori bounds on $\| u_n \|_{L^\infty}$ nor any particular behavior of the function $L$ at infinity. 
\end{remark}


\section{Non triviality of the functional}\label{sec:nontri}

In this section we seek to characterize when the functional $\f$ is trivial, \emph{i.e.} $\f\equiv+\infty$, in terms of properties of the nonlinear distortion $L$.

\begin{theorem}\label{prop:nontri}
The following two conditions are equivalent:
\begin{itemize}
\item[(i)] for every $f\in L^p(\o;\R^M)$ there exists $u\in \sp$ such that
$\f(u)<+\infty$, where $\f$ is the functional relative to $f$,
\item[(ii)] it holds
\begin{equation}\label{eq:linearinfty}
\limsup_{|t|\rightarrow+\infty}\frac{L(t)}{|t|}<+\infty\,.
\end{equation}
\end{itemize}
\end{theorem}

\begin{proof}
\emph{Step 1.} We start by proving the implication $(ii)\Rightarrow(i)$.
Let
\[
L_\infty:=\limsup_{|t|\rightarrow+\infty}\frac{L(t)}{|t|}\,.
\]
Fix $\varepsilon>0$ and let $\bar{t}>0$ be such that
\begin{equation}\label{eq:ineql}
L(t)\leq (L_\infty+\varepsilon)|t|\,,
\end{equation}
for every $t\in\R$ with $|t|\geq|\bar{t}|$. Let $K:=\max\{ L(t) \,:\, |t|\leq|\bar{t}| \}$.
We have
\begin{align}\label{eq:stimaL}
\int_D |L(f\cdot e) |^p\dx &= 
    \int_{D\cap \{|f\cdot e|<\bar{t} \}} |L(f\cdot e) |^p\dx
    +\int_{D\cap \{|f\cdot e|\geq\bar{t} \}} |L(f\cdot e) |^p\dx \nonumber\\
&\leq |D| K^p + (L_\infty+\varepsilon)^p\|f\|^p_{L^p}\,,
\end{align}
where in the last step we used \eqref{eq:ineql}. Taking $u=0$ and using \eqref{eq:stimaL} we conclude that
\begin{align*}
\f(u)&=\lambda\int_{\o\,\meno \,D}|f|^p\dx 
    + \mu\int_D \bigl|\, L(f\cdot e)  \,\bigr|^p\dx\\
&\leq \lambda\|f\|^p_{L^p}
    +\mu(\,|D| K^p + (L_\infty+\varepsilon)\|f\|^p_{L^p}\,)\,,
\end{align*}
and so $\f(u)<\infty$.

\emph{Step 2.} We now prove that $(i)\Rightarrow(ii)$. Assume that $(ii)$ fails, \emph{i.e.}, there exists $\{t_n\}_{n\in\N}$ with $|t_n|\geq1$,
$|t_n|\rightarrow\infty$, such that
\begin{equation}\label{eq:an}
\lim_{n\rightarrow+\infty}a_n=+\infty\,,\quad\quad\text{where }\quad a_n:=\frac{L^p(t_n)}{|t_n|^p}\,.
\end{equation}
Choose $n_k\nearrow\infty$ such that $a_{n_k}\geq 4^k$ for every $k\in\N$, and define
\[
b_n:=
\left\{
\begin{array}{ll}
2^{-k} & n=n_k\,,\\
0 & \text{otherwise}\,.
\end{array}
\right.
\]
Then it holds that
\begin{equation}\label{eq:bn}
\sum_{n=1}^\infty b_n=1\,,
\end{equation}
and
\begin{equation}\label{eq:anbn}
\sum_{n=1}^\infty a_nb_n = \sum_{k=1}^\infty a_{n_k}b_{n_k} \geq \sum_{k=1}^\infty 2^k = \infty\,. 
\end{equation}
Without loss of generality, we can assume that $e=e_1$, where $(e_1,\dots,e_N)$ is the canonical basis of $\R^N$, and that, in view of the fact that $\textrm{int}(D)\neq\emptyset$, $Q:=(0,1)^N\subset D$.
By \eqref{eq:bn}, and because $|t_n|\geq1$, it is possible to choose non-overlapping intervals $I_n\subset(0,1)$ such that
\begin{equation}\label{eq:in}
|I_n|=\frac{b_n}{|t_n|^p}\,.
\end{equation}
Define the function $f:\o\rightarrow\R^M$ as
\[
f(x):=(\widetilde{f}(x_1),0,\dots,0)\chi_{Q}(x)\,,
\]
where $x=(x_1,\dots,x_N)$, and $\widetilde{f}:\R\rightarrow\R$ is given by
\[
\widetilde{f}(s):=\sum_{n=1}^\infty t_n\chi_{I_n}(s)\,.
\]
Using \eqref{eq:in} and\eqref{eq:bn} in this order, we get
\[
\int_\o |f|^p \dx = \int_0^1 |\widetilde{f}(s)|^p \dd s =\sum_{n=1}^\infty |I_n||t_n|^p
    =\sum_{n=1}^\infty b_n =1\,.
\]
Let $u$ be an arbitrary function in $\sp$. We claim that
\[
\int_D |L(u\cdot e)-L(f\cdot e)|^p\dx=\infty\,.
\]
Write
\[
Q=(0,1)\times (0,1)^{N-1} \,.
\]
It is well known that for $\mathcal{L}^{N-1}$-a.e. $x'\in(0,1)^{N-1}$, the function
$u_{|_{x'}}: (0,1)\rightarrow\R$ given by
\[
u_{|_{x'}}(x):= u(x,x')
\]
is of bounded variation (see \cite[Section 3.11]{AFP}), thus bounded, and hence
\begin{equation}\label{eq:finiteu}
\int_0^1 |L(e\cdot u_{|_{x'}})|^p \, \dx<+\infty\,.
\end{equation}
On the other hand, using \eqref{eq:an}, \eqref{eq:in} and \eqref{eq:anbn} in this order, we have that
\begin{equation}\label{eq:inftyL}
\int_0^1 |L(\widetilde{f})|^p \,\dd s = \sum_{n=1}^\infty a_n|t_n|^p|I_n|
    =\sum_{n=1}^\infty a_n b_n = \infty\,.
\end{equation}
Invoking Tonelli's theorem and the inequality
\begin{equation}\label{eq:ineqp}
|a+b|^p\leq 2^{p-1}(|a|^p+|b^p|)\,,
\end{equation}
we get that
\begin{align*}
\int_D |L(&f\cdot e)-L(u\cdot e)|^p \dx \geq \int_Q |L(f\cdot e)-L(u\cdot e)|^p \dx \\
&=\int_{(0,1)^{N-1}} \left[\, \int_0^1 |L(\widetilde{f})-L(u_{|_{x'}}\cdot e)|^p \,\dd s  \,\right] \dd x'\\
&\hspace{0.05cm}\geq \int_{(0,1)^{N-1}} \left[\, 2^{1-p}\int_0^1 |L(\widetilde{f})|^p\dx 
    - \int_0^1 |L(u_{|_{x'}}\cdot e)|^p\dx\,\right] \dd x'\\
&=+\infty\,,
\end{align*}
where in the last step we used \eqref{eq:finiteu} and \eqref{eq:inftyL}. This proves that the functional $\f$ relative to $f$ is such that
\[
\f(u)=+\infty\,,
\]
for any $u\in\sp$, therefore $(i)$ fails.
\end{proof}

In the case $N=1$ it is possible to obtain a sharper result (see Remark \ref{rem:trivial}). Notice that, for $N=1$, we have $e\in\{\pm1\}$. For simplicity, we will assume $e=1$, so that $f\cdot e$ reduces to $f$.

\begin{proposition}
Assume that $N=1$, and let $f\in L^p(\o)$. Then, the following are equivalent:
\begin{itemize}
\item[(i)] there exists $u\in \sp$ such that $\f(u)<+\infty$, where $\f$ is the functional relative to $f$,
\item[(ii)] it holds
\[
\int_{D}|L(f)|^p\dx<+\infty\,.
\]
\end{itemize}
\end{proposition}

\begin{proof}
The validity of the implication $(ii)\Rightarrow (i)$ can be seen by taking $u=0$.
To show $(i)\Rightarrow(ii)$, we recall that a function of bounded variation in one dimension is bounded.
Since  $\o\subset\R$ is bounded, we have that
\begin{equation}\label{eq:finiteu1}
\int_D |L(u)|^p\dx<+\infty
\end{equation}
for every $u\in BV(\o;\R^M)$. Assume that
\[
\int_{D}|L(f)|^p\dx=+\infty\,.
\]
In view of $(i)$ choose $u\in\sp$ such that $\f(u)<+\infty$, and note that, by \eqref{eq:finiteu1},
\begin{align*}
\f(u)&\geq \lambda\int_D |L(f)-L(u)|^p\dx\\
&\geq 2^{1-p}\int_D |L(f)|^p\dx - \int_D |L(u)|^p\dx=+\infty\,,
\end{align*}
and we reached a contradiction.
\end{proof}

\begin{remark}\label{rem:trivial}
In the case $N>1$, while the implication $(ii)\Rightarrow(i)$ is clearly valid, it turns out that $(i)\Rightarrow(ii)$ is false.
Indeed, consider the function
\[
f(x):=
\left\{
\begin{array}{ll}
\displaystyle\frac{e}{|x|^\alpha} & |x|\leq 1\,,\\
&\\
0 & \text{otherwise}\,,
\end{array}
\right.
\]
for $\alpha\in(0,N-1)$, so that $f\in\sp$. Now take $L(s):=s^{\frac{N}{\alpha p}}$, $\o:=B(0,2)$ and $D:=B(0,1)$. Then
\[
\int_D |L(f\cdot e)|^p\dx=\infty\,,
\]
while we clearly have $\f(f)<\infty$.
\end{remark}


\section{Piecewise constant admissible functions}\label{sec:pc}

In this section we study the minimization problem for the functional $\f$ in two particular cases restricting the admissible class of minimizers: when we fix apriori a finite number of admissible colors, and when we fix apriori the number of colors that we are allowed to use but the color spectrum is not restricted.
We start with the former case.

\begin{definition}
Fix $k\in\N$ and let $A:=\{a_1,\dots, a_k\}\subset\R^M$. We define
\[
\mathcal{A}_A:=\left\{\, u\in\sp \,:\, u(x)\in A \,\text{ for a.e. } x\in\o  \,\right\}\,.
\]
A function $u\in\mathcal{A}_A$ will be written as
\begin{equation}\label{eq:u}
u=\sum_{i=1}^k a_i\chi_{\o_i}\,,
\end{equation}
where $\o_i=\o_i(1)$ (the points of density $1$ for $\o_i$, see Definition \ref{def:ptdensity}), for every $i=1,\dots,k$.
\end{definition}

Consider the minimization problem
\begin{equation}\label{eq:minpbfinite}
\min_{u\in\mathcal{A}_A}\f(u)\,.
\end{equation}

\begin{theorem}\label{thm:kfixedcolors}
Let $\o\subset\R^N$ be an open connected set with Lipschitz boundary. Let $p\geq1$ and $f\in L^p(\o;\R^M)$ be such that $L(f\cdot e)\in L^p(\o)$.
Then, the minimization problem \eqref{eq:minpbfinite} admits a solution.
Assume, in addition, that, for some $q\geq N(p-1)$,
\begin{itemize}
\item[(H1)] $f\in L^{q}(\o;\R^M)$,
\item[(H2)] $L(f\cdot e)\in L^{q}(\o;\R^M)$,
\item[(H3)] $|a_i-a_j|<|a_i-a_k|+|a_k-a_j|$ whenever the indexes $i,j,k$ are different,
\end{itemize}
Let $u\in\mathcal{A}_A$ be a solution of \eqref{eq:minpbfinite}. Then,
\begin{itemize}
\item[(i)] $\o_i$ is open, for every $i=1,\dots,k$, and thus $J_u\cap\o=\overline{J_u}\cap\o$,
\item[(ii)] each $\partial^*\Omega_i$ is the union of relatively open sets of class $C^{1,\alpha}$, where $\alpha:=\frac{1}{2}\left(1-\frac{N(p-1)}{q}\right)$ and a closed singular set of $\hno$ measure zero,
\item[(iii)] there exists $\beta>0$, depending on $\o, D, N, \lambda, \mu, \|f\|_{L^{q}}, \|L(f\cdot e)\|_{L^{q}}$ and $\mathcal{A}$, such that
\[
\liminf_{\rho\to 0}\frac{\hno(J_u\cap B_\rho(x))}{\rho^{N-1}}\geq\beta\,,
\]
for every $x\in \overline{J_u}\cap\o$.
\end{itemize}
\end{theorem}

\begin{remark}
Notice that condition $(H2)$ is automatically satisfied if $(H1)$ and \eqref{eq:linearinfty} hold.

The hypothesis $(H3)$ requires triples of $a_i$'s to not be aligned, and is believed to be necessary for having regularity (see \cite{Leo}). From the technical point of view, it is needed in order to prove Lemma \ref{lem:mincut}.

Notice that hypotheses $(H1), (H2)$ and $(H3)$ are not needed for obtaining existence of a solution to the minimization problem \eqref{eq:minpbfinite} (see Proposition \ref{prop:existence}).
\end{remark}

The general strategy we adopt to prove Theorem \ref{thm:kfixedcolors} is similar to the one used by Tamanini and Congedo in \cite{TamCon}.
The idea is to recast the minimization problem \eqref{eq:minpbfinite} in the setting of Caccioppoli partitions.

\begin{definition}\label{def:Cacc}
A \emph{$k$-finite Caccioppoli partition} of $\o$ is a finite collection $\mathcal{U}=(U_1,\dots,U_k)$
of measurable subsets of $\o$ satisfying the following properties:
\begin{itemize}
\item[(i)] each $U_i$ has finite perimeter in $\o$,
\item[(ii)] $U_i=U_i(1)$,
\item[(iii)] $|U_i\cap U_j|=0$ if $i\neq j$,
\item[(iv)] $\left|\, \o\meno\cup_{i=1}^k U_i \,\right|=0$.
\end{itemize}
Define the \emph{perimeter} of the partition $\mathcal{U}$ in $\o$ as
\[
\mathrm{Per}(\mathcal{U};\o):=\hno\left(\, \bigcup_{i=1}^k \rb U_i\cap\o  \,\right)\,.
\]
Denote by $\mathcal{C}_k(\o)$ the family of all $k$-finite Caccioppoli partitions of $\o$.
\end{definition}

\begin{remark}
Notice that condition $(ii)$ of Definition \ref{def:Cacc} is imposed to guarantee a well defined representative of each measurable set in the partition.
Indeed, it holds that $A(1)=B(1)$ for every sets $A,B\subset\R^N$ with $|A \triangle B|=0$,
where $A\triangle B:=(A\setminus B)\cup (B\setminus A)$.
Moreover, conditions $(iii)$ and $(iv)$ assert that the $U_i$'s are pairwise disjoint and cover $\o$, in a measure theoretical sense.
\end{remark}

We recall the following structure theorem for $k$-finite Caccioppoli partitions (for a proof, see \cite{ConTam}).

\begin{theorem}\label{thm:struct}
Let $\mathcal{U}=(U_1,\dots,U_k)$ be a $k$-finite Caccioppoli partition of $\o$. Then
\[
\hno\left[\, \o\setminus \left(\, \bigcup_{i=1}^k U_i \,\cup\,
    \bigcup_{i<j=1}^k \left[ U_i\left(1/2\right)\cap U_j\left(1/2\right) \right]  \,\right)  \,\right]=0\,,
\]
where we recall that $U_i(1/2)$ denotes the sets of points of density $1/2$ of $U_i$ (see Definition \ref{def:ptdensity}).
In particular, the perimeter of the interfaces is given by
\begin{equation}\label{eq:perinter}
\mathrm{Per}(\mathcal{U};\o)=\frac{1}{2}\sum_{i=1}^k \per(U_i\,;\,\o)
    =\frac{1}{2}\sum_{i=1}^k \sum_{j\neq i}\hno(\rb U_i\cap \rb U_j\cap \o)\,.
\end{equation}
\end{theorem}

In order to recast our minimization problem in terms of Caccioppoli partitions, observe that for $u\in\mathcal{A}_A$ it holds
\begin{align*}
\f(u)&=\sum_{i=1}^k\sum_{j>i} |a_i-a_j|\hno(\rb\o_i\cap\rb\o_j\cap\o) + \lambda\sum_{i=1}^k \int_{\o_i\,\meno\, D} |a_i-f|^p\dx\\
&\hspace{0.5cm}+ \mu\sum_{i=1}^k \int_{\o_i\cap D} |L(a_i\cdot e)-L(f\cdot e)|^p\dx\,.
\end{align*} 

We are led to the following definition.

\begin{definition}
Let $\lambda,\mu>0$, $f\in L^p(\o;\,\R^M)$ and $L:\R\to\R$ be a continuous function.
We define the functional $\g:\mathcal{C}_k\rightarrow[0,\infty]$ as
\begin{align*}
\g(\mathcal{U})&:=\sum_{i=1}^k\sum_{j>i} \alpha_{ij}\hno(\rb U_i\cap\rb U_j\cap\o)
    + \lambda\sum_{i=1}^k \int_{U_i\,\meno\, D} |a_i-f|^p\dx\\
&\hspace{1cm}+ \mu\sum_{i=1}^k \int_{U_i\cap D} |L(a_i\cdot e)-L(f\cdot e)|^p\dx\,.
\end{align*}
where $\alpha_{ij}:=|a_i-a_j|>0$, for $i\in\{1,\dots,k\}$ and $j>i$, and $\mathcal{U}=(U_1,\dots,U_k)$.
Moreover, for $\mathcal{U}\in\mathcal{C}_k(\o)$ and a Borel set $A\subset\o$ we define the weighted perimeter of the partition $\mathcal{U}$ in $A$ as
\begin{equation}\label{eq:ptilde}
\mathrm{WPer}(\mathcal{U}\,;A):=\sum_{i=1}^k\sum_{j>i} \alpha_{ij}\hno(\rb U_i\cap\rb U_j\cap A)\,.
\end{equation}
\end{definition}

We then consider the minimization problem
\begin{equation}\label{eq:minpbCacc}
\min_{\mathcal{U}\in\mathcal{C}_k}\g(\mathcal{U})\,.
\end{equation}

The  equivalence between the minimization problems \eqref{eq:minpbfinite} and \eqref{eq:minpbCacc} now follows.

\begin{lemma}\label{lem:equiv}
Let $u\in\mathcal{A}_A$ be a solution to the minimization problem \eqref{eq:minpbfinite}.
Then $\mathcal{O}:=(\o_1,\dots,\o_k)\in\mathcal{C}_k$ is a solution to the minimization problem \eqref{eq:minpbCacc}, where $\o_1,\dots,\o_k$ are given by \eqref{eq:u}.

Conversely, if $\mathcal{U}:=(U_1,\dots,U_k)\in\mathcal{C}_k$ is a solution to the minimization problem \eqref{eq:minpbCacc}, then the function
\[
u:=\sum_{i=1}^k a_i\chi_{U_i}
\]
belongs to $\mathcal{A}_A$ and is a solution to the minimization problem \eqref{eq:minpbfinite}.
\end{lemma}

We now focus on the study of the minimization problem \eqref{eq:minpbCacc}.

\begin{proposition}\label{prop:existence}
Assume $(H3)$ of Theorem \ref{thm:kfixedcolors} holds true. Then the minimization problem \eqref{eq:minpbCacc} admits a solution.
\end{proposition}

\begin{proof}
Let $\{\mathcal{U}_n\}_{n\in\N}\subset\mathcal{C}_k$ be a minimizing sequence, where $\mathcal{U}_n:=(U^n_1,\dots,U^n_k)$.
Without loss of generality, we can assume that $\sup_{n\in\N}\g(\mathcal{U}_n)<\infty$.
Since, for every $i=1,\dots,k$ and every $n\in\N$ it holds that
\[
\per(U^n_i\,;\o)=\hno(\rb U^n_i\cap \o)= \sum_{j\neq i} \hno(\rb U^n_j \cap \rb U^n_i \cap \o)\,,
\]
using the fact that $\alpha_{ij}>C$ for some $C>0$, and that
\begin{equation}\label{eq:equper}
\sum_{i=1}^k\sum_{j>i} \hno(\rb U^n_i\cap\rb U^n_j\cap\o)=\hno\left(\, \bigcup_{i=1}^k \rb U^n_i \cap \o \,\right)\,,
\end{equation}
we get $\sup_n\per(U^n_i\,;\,\o)<\infty$ for all $i=1,\dots,k$.
A diagonalization argument yields a subsequence (not relabeled), verifying
\begin{equation}\label{eq:lscp}
\chi_{U^n_i}\to \chi_{U_i} \,\,\text{ in } L^1\,,\quad\quad \per(U_i\,;\,\o)\leq \liminf_{n\to\infty}\per(U^n_i\,;\,\o)\,,
\end{equation}
for some sets $U_1,\dots,U_k$ of finite perimeter in $\o$.
By replacing, if needed, each $U_i$ with $U_i(1)$, it is easy to see that $\mathcal{U}=(U_1,\dots,U_k)\in\mathcal{C}_k$.
Using the lower semi-continuity result by Ambrosio and Braides (see \cite[Theorem 2.1]{AmbBra2}), we obtain 
\begin{align*}
\mathrm{WPer}(\mathcal{U}\,;\,\o)\leq\liminf_{n\to\infty}\,\,\mathrm{WPer}(\mathcal{U}_n\,;\,\o)\,.
\end{align*}
Finally, by Lebesgue's dominated convergence theorem, we deduce that
\[
\int_{U^n_i\,\meno\, D} |a_i-f|^p\dx\to\int_{U_i\,\meno\, D} |a_i-f|^p\dx\,,
\]
as $n\to\infty$, and, by Fatou's lemma, that
\[
\int_{U_i\cap D} |L(a_i\cdot e)-L(f\cdot e)|^p\dx\leq\liminf_{n\to\infty}\int_{U^n_i\cap D} |L(a_i\cdot e)-L(f\cdot e)|^p\dx\,,
\]
for every $i=1,\dots,k$.
\end{proof}

The following elimination theorem is the fundamental tool we will use to establish regularity properties of solutions to the minimization problem \eqref{eq:minpbCacc}. 
Our result extends the one proved by Leonardi in \cite{Leo} for the functional $\mathcal{G}$ with $\lambda=\mu=0$.

\begin{theorem}\label{thm:elimination}
Let $\mathcal{U}=(U_1,\dots,U_k)\in\mathcal{C}_k$ be a solution to the minimization problem \eqref{eq:minpbCacc}, and assume that hypotheses $(H1)$, $(H2)$ and $(H3)$ of Theorem \ref{thm:kfixedcolors} hold.
Set $V:=U_3\cup\dots \cup U_k$.
Then there exist $\eta,r_0>0$ such that if $x\in\o$, $B_{r_0}(x)\subset\subset\o$, and if
\begin{equation}\label{eq:smallness}
|V\cap B_r(x)|\leq\eta r^N
\end{equation}
for some $0<r<r_0$, then $|V\cap B_{r/2}(x)|=0$.
\end{theorem}

\begin{remark}\label{rem:perm}
Theorem \ref{thm:elimination} holds also in the case where $\mathcal{U}$ is a local minimizer of $\mathcal{G}$, \emph{i.e.}, if there exists a ball $B_{R}(\bar{x})\subset\subset\o$
\[
\g(\mathcal{U})\leq\g(\mathcal{V})
\]
for every $\mathcal{V}\in\mathcal{C}_k$ with $U_i\triangle V_i\subset B_{R}(\bar{x})$, for $i=1,\dots,k$.
Moreover, the result continues to be satisfied when
\[
V:=U_{\sigma(1)}\cup\dots \cup U_{\sigma(k-2)}\,,
\]
where  $\sigma:\{1,\dots,k\}\to\{1,\dots,k\}$ is a permutation.
\end{remark}

Our strategy to prove Theorem \ref{thm:elimination} is similar to the one used by Leonardi in \cite{Leo}.
The idea is the following: let $\mathcal{U}=(U_1,\dots,U_k)$ be as in the statement of the theorem.
For a.e. $r\in(0,r_0)$ we seek for a variation $\mathcal{\widetilde{U}}$ of $\mathcal{U}$ such that the difference of the energy of the two configurations is controlled by the perimeter of $V$.
The family of perturbations we consider is the one where we locally divide the partition $\mathcal{U}$ in two classes, \emph{i.e.}, we consider $\mathcal{U}_1:=\{U_i\}_{i\in I}$ and $\mathcal{U}_2:=\{U_i\}_{i\in \{1,\dots,k\}\setminus I}$, for some set of indexes $I\subset\{1,\dots,k\}$ with $1\in I$, $2\not\in I$, and we glue together all the sets in the first class with $U_1$ and all the 
others with $U_2$. To be more precise, following \cite{Leo}, we introduce the following notation.

\begin{definition}
Let $k\in\N$ and let $I_1\subset\{1,\dots,k\}$ with $1\in I_1$ and $2\not\in I_1$.
Set $I_2:=\{1,\dots,k\}\setminus I_1$.
For $x\in\o$ let $r>0$ be such that $B_r(x)\subset\o$.
If $\mathcal{U}=(U_1,\dots,U_k)\in\mathcal{C}_k(\o)$, define $\mathcal{U}_{I_1}^r=(\widetilde{U}_1,\dots,\widetilde{U}_k)\in\mathcal{C}_k(\o)$ as
\begin{equation}\label{eq:defvar}
\widetilde{U}_i:=
\left\{
\begin{array}{ll}
U_i\setminus B_r(x) & \text{ if } i>2\,,\\
&\\
U_i \cup \displaystyle\bigcup_{j\in I_i} \left(U_j\cap B_r(x)\right) & \text{ if } i=1,2\,.\\
\end{array}
\right.
\end{equation}
Moreover, define
\[
\Delta_r \mathrm{WPer}(\mathcal{U}):=\mathrm{WPer}(\mathcal{U}^r_{I_1}; B_r(x))
    -\mathrm{WPer}(\mathcal{U}; B_r(x))\,.
\]
\end{definition}

In order to prove the elimination Theorem \ref{thm:elimination}, we need to invoke a result proved by Leonardi in \cite{Leo}.
This is the technical point where condition $(H3)$ of Theorem \ref{thm:elimination} is needed.

\begin{lemma}\label{lem:mincut}
Let $\mathcal{U}=(U_1,\dots,U_k)\in\mathcal{C}_k(\o)$ be a solution to the minimization problem \eqref{eq:minpbCacc}, and assume that hypothesis $(H3)$ of Theorem \ref{thm:kfixedcolors} holds.
Let $x\in\o$ and $R>0$ be such that $B_R(x)\subset\subset\o$.
Then, for almost all $r\in(0,R)$ there exists $I^r_1\subset\{1,\dots,k\}$ with $1\in I^r_1$, and $2\not\in I^r_1$, such that
\begin{equation}\label{eq:mincut}
\Delta_r\mathrm{WPer}(\mathcal{U})\leq -C_1\,\per(V; B_r(x))\,,
\end{equation}
where $V:=U_3\cup\dots\cup U_k$ and $C_1>0$ is a constant depending only on $\mathcal{U}$.
\end{lemma}

Lemma \ref{lem:mincut} allows us to prove the elimination property for a solution of the minimization problem \eqref{eq:minpbCacc}.

\begin{proof}[Proof of Theorem \ref{thm:elimination}]
Let $r_0>0$ be such that $B_{r_0}(x)\subset\subset\o$. For $r\in(0,r_0)$ define $\alpha(r):=|V\cap B_r(x)|$.
Then $\alpha$ is a non-decreasing function. Using the coarea formula (see \cite[Theorem 2.93]{AFP}) and the differentiability a.e. of monotone functions we obtain that for a.e. $r\in(0,r_0)$ it holds
\begin{equation}\label{eq:der}
\alpha'(r)=\int_{\partial B_r(x)} \chi_V\dhno\,.
\end{equation}
Since for every $i\in\{1,\dots,k\}$ the set $U_i$ has finite perimeter in $\o$, for a.e. $r\in(0,r_0)$ we have that
\begin{equation}\label{eq:noper}
\per(U_i\,;\,\partial B_r(x))=0\,.
\end{equation}
For a.e. $r\in(0,r_0)$ we have that \eqref{eq:der} and \eqref{eq:noper} hold and that Lemma \ref{lem:mincut} provides a set of indexes $I^r_1$ relative to $\mathcal{U}$ for which \eqref{eq:mincut} is satisfied.
Fix one of these $r\in(0,r_0)$.

Our goal is to get an estimate of $\g(\mathcal{U}^r_{I^r_1})-\g(\mathcal{U})$ in terms of $\alpha(r)$ and $\alpha'(r)$.\\

\emph{Step 1. Estimate of the weighted perimeter.}
 We have
\begin{align}\label{eq:firstineq}
\Delta_{r_0}\,\mathrm{WPer}(\mathcal{U})&=\mathrm{WPer}(\mathcal{U}^r_{I^r_1}\,;\, B_{r_0}(x))
    -\mathrm{WPer}(\mathcal{U}\,;\, B_{r_0}(x)) \nonumber \\
&=\mathrm{WPer}(\mathcal{U}^r_{I^r_1}\,;\, B_{r_0}(x)\setminus \bar{B}_r(x))
    +\mathrm{WPer}(\mathcal{U}^r_{I^r_1}\,;\, \partial B_r(x)) \nonumber\\
&\hspace{0.6cm} +\mathrm{WPer}(\mathcal{U}^r_{I^r_1}\,;\, B_r(x))  
-\mathrm{WPer}(\mathcal{U}\,;\, B_{r_0}(x)\setminus \bar{B}_r(x)) \nonumber\\
&\hspace{0.6cm} -\mathrm{WPer}(\mathcal{U}\,;\, \partial B_r(x))
    -\mathrm{WPer}(\mathcal{U}\,;\, B_r(x))\,.
\end{align}
The fact that $\mathcal{U}^r_{I^r_1}$ and $\mathcal{U}$ coincide in $B_{r_0}(x)\setminus \bar{B}_r(x)$ yields
\begin{equation}\label{eq:equal}
\mathrm{WPer}(\mathcal{U}^r_{I^r_1}\,;\, B_{r_0}(x)\setminus \bar{B}_r(x))
    -\mathrm{WPer}(\mathcal{U}\,;\, B_{r_0}(x)\setminus \bar{B}_r(x))=0\,.
\end{equation}
Moreover, by \eqref{eq:noper} we have that (see \cite[Remark 3.57]{AFP})
\begin{equation}\label{eq:alphaprime}
\mathrm{WPer}(\mathcal{U}^r_{I^r_1}\,;\, \partial B_r(x))\leq M\alpha'(r)\,,
\end{equation}
where $M:=\max_{i,j}\alpha_{ij}$.
In view of \eqref{eq:firstineq}, \eqref{eq:equal}, and \eqref{eq:alphaprime}, we get
\begin{equation}\label{eq:estwideper}
\Delta_{r_0}\,\mathrm{WPer}(\mathcal{U})\leq \Delta_r\mathrm{WPer}(\mathcal{U})+M\alpha'(r) \,.
\end{equation}

\emph{Step 2. Estimate of the volume terms.} Using the inequality (see \cite[Proposition 4.64]{FonLeo})
\begin{equation}\label{eq:ineq1}
\bigl| |a|^p -|b|^p \bigr|\leq 2^{p-1}p|a-b|(|a|^{p-1}+|b|^{p-1})\,,
\end{equation}
and the definition of $\widetilde{U}_i$ (see \eqref{eq:defvar}), for $i=1,2$ we have that 
\begin{align}\label{eq:estid}
&\int_{\widetilde{U}_i\cap B_r(x)\setminus D} |a_i-f|^p \dx - \sum_{j\in I_i^r} \int_{U_j\cap B_r(x)\setminus D} |a_j-f|^p \dx \nonumber\\
&\hspace{1cm}=\sum_{j\in I_i^r} \int_{U_j\cap B_r(x)\setminus D} \left(\, |a_i-f|^p - |a_j-f|^p \,\right) \dx \nonumber \\
&\hspace{1cm}\leq 2^{p-1}p\sum_{j\in I_i^r} |a_i-a_j|\int_{U_j\cap B_r(x)\setminus D} \left(\, |a_i-f|^{p-1}+|a_j-f|^{p-1} \,\right) \dx
    \nonumber\\
&\hspace{1cm}\leq \mathrm{diam}(A)2^{p-1}p\sum_{j\in I_i^r} \Bigl(\, \|a_i-f\|_{L^{N(p-1)}(U_j\cap B_{r_0}(x)\setminus D;\R^M)} \nonumber \\
&\hspace{1.6cm}+\|a_j-f\|_{L^{N(p-1)}(U_j\cap B_{r_0}(x)\setminus D;\R^M)}  \,\Bigr)|U_j\cap B_r(x)\setminus D|^{\frac{N-1}{N}}\,,
\end{align}
where in the last step we used H\"older inequality together with the fact that $f\in L^{q}(\o;\R^M)$ with $q\geq N(p-1)$ (see hypothesis $(H1)$).
Here $\mathrm{diam}(A):=\max\{|a_r-a_s| \,:\, r,s\in\{1,\dots,k\}\}$
denotes the diameter of the set $A$.
Similarly, we deduce that
\begin{align}\label{eq:estod}
&\int_{\widetilde{U}_i\cap B_r(x)\cap D} |L(a_i\cdot e)-L(f\cdot e)|^p \dx - \sum_{j\in I_i^r} \int_{U_j\cap B_r(x)\cap D} |L(a_j\cdot e)-L(f\cdot e)|^p \dx \nonumber\\
&\hspace{1cm}\leq \mathrm{diam}(L(A\cdot e))2^{p-1}p\sum_{j\in I_i^r} \Bigl(\, \|L(a_i\cdot e)-L(f\cdot e)\|_{L^{N(p-1)}(U_j\cap B_{r_0}(x)\cap D;\R^M)} \nonumber \\
&\hspace{2cm}+\|L(a_j\cdot e)-L(f\cdot e)\|_{L^{N(p-1)}(U_j\cap B_{r_0}(x)\cap D;\R^M)}  \,\Bigr)|U_j\cap B_r(x)\cap D|^{\frac{N-1}{N}}\,,
\end{align}
where we have used the fact that $L(f\cdot e)\in L^{q}(\o;\R^M)$ with $q\geq N(p-1)$ (see hypothesis $(H2)$).
Here $\mathrm{diam}(L(A\cdot e)):=\max\{ |L(a_r\cdot e)-L(a_s\cdot e)| \,:\, r,s\in\{1,\dots,k\} \}$.
Thus, we obtain that there exists a constant $C_2(r_0)$, with $C_2(r_0)\to0$
as $r_0\to0$ such that, for $i=1,2$, it hold
\begin{align}\label{eq:estid1}
&\int_{\widetilde{U}_i\cap B_r(x)\setminus D} |a_i-f|^p \dx - \sum_{j\in I_i^r} \int_{U_j\cap B_r(x)\setminus D} |a_j-f|^p \dx \nonumber\\
&\hspace{1cm}\leq C_2(r_0)\sum_{j\in I_i^r} |U_j\cap B_r(x)\setminus D|^{\frac{N-1}{N}}\,,
\end{align}
and
\begin{align}\label{eq:estod1}
&\int_{\widetilde{U}_i\cap B_r(x)\cap D} |L(a_i\cdot e)-L(f\cdot e)|^p \dx  \nonumber\\
&\hspace{4cm} - \sum_{j\in I_i^r} \int_{U_j\cap B_r(x)\cap D} |L(a_j\cdot e)-L(f\cdot e)|^p \dx \nonumber\\
&\hspace{2cm}\leq C_2(r_0)\sum_{j\in I_i^r} |U_j\cap B_r(x)\cap D|^{\frac{N-1}{N}}\,.
\end{align}
Using \eqref{eq:defvar}, \eqref{eq:estid1} we get
\begin{align}\label{eq:ineqvolume1}
&\lambda\sum_{j=1}^k \left[\, \int_{\widetilde{U}_j\cap B_r(x)\,\meno\, D} |a_j-f|^p\dx
    -\int_{U_j\cap B_r(x)\,\meno\, D} |a_j-f|^p\dx \,\right] \nonumber \\
&\hspace{1cm}=\lambda\sum_{i=1,2}\left[\, \int_{\widetilde{U}_i\cap B_r(x)\setminus D} |a_i-f|^p \dx - \sum_{j\in I_i^r} \int_{U_j\cap B_r(x)\setminus D} |a_j-f|^p \dx \,\right] \nonumber \\
&\hspace{1cm}\leq\lambda C_2(r_0)\sum_{i=1,2} \sum_{j\in I^r_i}|U_j\cap B_r(x)\meno D|^{\frac{N-1}{N}} \nonumber\\
&\hspace{1cm}\leq\lambda C_2(r_0)k^{\frac{1}{N}}|V\cap B_r(x)\meno D|^{\frac{N-1}{N}}\,,
\end{align}
where in the last step we used the definition of $V$ and the inequality
\begin{equation}\label{eq:ineqconc}
\sum_{i=1}^k |p_i|^{\frac{N-1}{N}}\leq k^{\frac{1}{N}}\left(\, \sum_{i=1}^k |p_i| \,\right)^{\frac{N-1}{N}}\,,
\end{equation}
that results from the concavity of the function $|p|\mapsto |p|^{\frac{N-1}{N}}$.
With a similar argument, and by \eqref{eq:estod1} and \eqref{eq:ineqconc}, we obtain
\begin{align}\label{eq:ineqvolume2}
&\mu\sum_{j=1}^k \Biggl[\, \int_{\widetilde{U}_j\cap B_r(x)\,\cap\, D} |L(a_j\cdot e)-L(f\cdot e)|^p\dx \nonumber\\
&\hspace{4cm}-\int_{U_j\cap B_r(x)\,\cap\, D} |L(a_j\cdot e)-L(f\cdot e)|^p\dx \,\Biggr] \nonumber \\
&\hspace{2cm}\leq \mu C_2(r_0) k^{\frac{1}{N}}|V\cap B_r(x)\cap D|^{\frac{N-1}{N}}\,.
\end{align}
Thus, \eqref{eq:ineqvolume1} and \eqref{eq:ineqvolume2} yield
\begin{align}\label{eq:ineqvolume}
&\lambda\sum_{j=1}^k \left[\, \int_{\widetilde{U}_j\cap B_r(x)\,\meno\, D} |a_j-f|^p\dx
    -\int_{U_j\cap B_r(x)\,\meno\, D} |a_j-f|^p\dx \,\right] \nonumber \\
&\hspace{2cm}+\mu\sum_{j=1}^k \Biggl[\, \int_{\widetilde{U}_j\cap B_r(x)\,\cap\, D} |L(a_j\cdot e)-L(f\cdot e)|^p\dx \nonumber \\
&\hspace{3cm}-\int_{U_j\cap B_r(x)\,\cap\, D} |L(a_j\cdot e)-L(f\cdot e)|^p\dx \,\Biggr] \nonumber \\
&\hspace{1cm}\leq C_3(r_0) |V\cap B_r(x)|^{\frac{N-1}{N}} \nonumber\\
&\hspace{1cm}= C_3(r_0)\alpha(r)^{\frac{N-1}{N}}\,,
\end{align}
where we used again inequality \eqref{eq:ineqconc}. Here $C_3(r_0)\to0$ as $r_0\to0$.\\

\emph{Step 3. Conclusion.}
The minimality of $\mathcal{U}$, together with \eqref{eq:mincut}, \eqref{eq:estwideper} and \eqref{eq:ineqvolume}, yields
\begin{align*}\label{eq:firstest}
0&\leq \g(\mathcal{U}^r_{I^r_1})-\g(\mathcal{U})\\
&\leq \Delta_r\mathrm{WPer}(\mathcal{U})+M\alpha'(r) + C_3(r_0)\alpha(r)^{\frac{N-1}{N}}\\
&\leq -C_1\per(V\,;\, B_r(x))+M\alpha'(r)+ C_3(r_0)\alpha(r)^{\frac{N-1}{N}}\\
&\leq -C_1N\omega_N^{\frac{1}{N}}\alpha(r)^{\frac{N-1}{N}}+M\alpha'(r)+ C_3(r_0)\alpha(r)^{\frac{N-1}{N}}\,,
\end{align*}
where the last inequality follows from the isoperimetric inequality and the definition of $\alpha$.
Here $C_1$ is the constant given by Lemma \ref{lem:mincut}.
We then deduce that
\[
(\alpha^{\frac{1}{N}}(r))'=\frac{1}{N}\alpha'(r)\alpha^{\frac{1-N}{N}}(r)\geq \frac{C_1N\omega_N^{\frac{1}{N}}-C_3(r_0)}{MN}\,,
\]
and integrating this inequality from $r/2$ to $r$ yields
\begin{equation}\label{eq:estalpha}
\alpha^{\frac{1}{N}}(r)-\alpha^{\frac{1}{N}}\left(\frac{r}{2}\right)\geq\frac{r}{2}\frac{C_1N\omega_N^{\frac{1}{N}}-C_3(r_0)}{MN}\,.
\end{equation}
Choosing $r_0$ sufficiently small in such a way that $C_3(r_0)<\frac{C_1 N\omega_N^{\frac{1}{N}}}{2}$, from
\eqref{eq:estalpha} we get
\[
\alpha^{\frac{1}{N}}(r)-\alpha^{\frac{1}{N}}\left(\frac{r}{2}\right)\geq \frac{r}{2}\eta^{\frac{1}{N}}\,,
\]
where we set
\[
\eta:=\left(\, \frac{C_1\omega_N^{\frac{1}{N}}}{4M} \,\right)^N\,.
\]
In view of \eqref{eq:smallness} we now take $0<r<r_0$ such that $\alpha(r)\leq\eta r^N$, by \eqref{eq:estalpha} we get
\[
\alpha\left(\frac{r}{2}\right)\leq0\,.
\]
Since $\alpha(r/2)\geq0$, we conclude that $\alpha(r/2)=0$.
\end{proof}

The proof of Theorem \ref{thm:kfixedcolors} hinged on two results.
The first is a general isoperimetric inequality (see \cite[Lemma 4.2]{TamCon}).

\begin{lemma}\label{lem:genisop}
There exist two constants $\gamma_1,\gamma_2>0$, depending only on $N$, with the following property:
consider a ball $B_r\subset\R^N$, a finite set $A\subset\R^M$, and let $u\in BV(B_r;A)$ satisfy
\[
\hno(J_u\cap B_r)<\gamma_1\, r^{N-1}\,.
\]
Then there exists $i\in\{1,\dots,k\}$ such that
\[
|B_r\setminus \o_i|^{\frac{N}{N-1}}\leq \gamma_2\,\hno(J_u\cap B_r)\,,
\]
where we write $u$ as in \eqref{eq:u}.
\end{lemma}

The second result is a well-known regularity property of almost-minimal sets, due to Tamanini (see \cite[Theorem 1]{Tam}).

\begin{theorem}\label{thm:regTam}
Let $U\subset\R^N$ be an open set, and let $E\subset\R^N$ be a set of finite perimeter with the following property:
there exist constants $C>0$, $R>0$, and $\alpha\in(0,1)$, such that for every $x\in U$ and every $r\in(0,R)$, it holds
\[
\hno(\rb E \cap B_r(x)\cap U)\leq \hno(\rb F\cap B_r(x)\cap U)+Cr^{N-1+2\alpha}\,,
\]
for every set $F\subset\R^N$ of finite perimeter with $F\triangle E\subset\subset B_r(x)$.
Then $\rb E\cap U$ is a $C^{1,\alpha}$-hypersurface up to a closed $\hno$-negligible set.
\end{theorem}

We are now in position to prove Theorem \ref{thm:kfixedcolors}.

\begin{proof}[Proof of Theorem \ref{thm:kfixedcolors}]
The existence of a solution $u\in\mathcal{A}_A$ to the minimization problem \eqref{eq:minpbfinite} follows from Lemma \ref{lem:equiv} and Proposition \ref{prop:existence}.

Let $\mathcal{U}=(\o_1,\dots,\o_k)\in\mathcal{C}_k(\o)$ be the corresponding solution of the minimization problem \eqref{eq:minpbCacc} given by Lemma \ref{lem:equiv}, where we write
\[
u=\sum_{i=1}^k a_i\chi_{\o_i}\,.
\]

\emph{Step 1: Proof of $(i)$.} Assume that $(H1)$, $(H2)$ and $(H3)$ hold.
By definition of Caccioppoli partition, for every $i=1,\dots,k$, we have that $\o_i$ coincides with its set of points of density $1$.
Let $\eta$ be the constant given by Theorem \ref{thm:elimination}. Then, for every $x\in \o_i$ it is possible to find $r>0$ such that
\[
|B_r(x)\setminus \o_i|<\eta r^N\,.
\]
Applying Theorem \ref{thm:elimination} and using Remark \ref{rem:perm}, we get that $|V\cap B_{r/2}(x)|=0$ for every
\[
V:=\bigcup_{j\in \{1,\dots,k\}\setminus\{i\}} \o_{\sigma(j)}\,,
\]
where $\sigma:\{1,\dots,k\}\setminus\{i\}\to\{1,\dots,k\}\setminus\{i\}$ is a permutation.
Thus, we obtain that
\begin{equation}\label{eq:denszeroball}
|B_{r/2}(x)\cap \o_j|=0
\end{equation}
for all $j\in\{1,\dots,k\}\setminus\{i\}$.
Assume there exists $y\in B_{r/2}(x)$ such that $y\not\in\o_i$. Since $\o_i=\o_i(1)$, there exists $j\in\{1,\dots,k\}\meno\{i\}$ and a sequence $\{\rho_n\}_{n\in\N}$ with $\rho_n\to0$ as $n\to\infty$ such that
\[
\lim_{n\to\infty}\frac{|\o_j\cap B_{\rho_n}(x)|}{\omega_N \rho_n^N}>0\,.
\]
This contradicts \eqref{eq:denszeroball}. Thus $B_{r/2}(x)\subset\o_i$, and in turn $\o_i$ is open.
In particular, we conclude that $J_u$ is a closed set.\\

\emph{Step 2: Proof of $(ii)$.} Since $u$ is of bounded variation, by a standard result (see \cite[Theorem 3.78]{AFP}) we have that $\hno$-a.e. $x\in J_u$ belongs to $\o_i(1/2)\cap \o_j(1/2)$, for just one pair of indexes $i,j$.
Fix
\[
\bar{x}\in J_u\cap\o_i(1/2)\cap \o_j(1/2)\setminus \bigcup_{l\neq i,j} \o_l(1/2)\,.
\]
Using Definition \ref{def:ptdensity} we have that
\[
\lim_{\rho\to0}\frac{|\o_i\cap B_{\rho}(\bar{x})|}{\omega_N \rho^N}=\frac{1}{2}\,,\quad\quad
    \lim_{\rho\to0}\frac{|\o_j\cap B_{\rho}(\bar{x})|}{\omega_N \rho^N}=\frac{1}{2}\,.
\]
Thus, there exists $\rho>0$ such that
\[
|\o_s\cap B_{\rho}(\bar{x})|\leq \eta \rho^N\,,
\]
for all $s\in\{1,\dots,k\}\meno\{i,j\}$, where $\eta>0$ is the constant given by Theorem \ref{thm:elimination}.
Setting $r:=\rho/2$, and arguing as we did in Step 1, we get that $B_r(\bar{x})\subset \o_i\cup \o_j\cup J_u$.

We claim that there exists a constant $C>0$ such that
\begin{equation}\label{eq:minloc}
\hno(\rb\o_i\,;\, B_r(\bar{x}))\leq \hno(\rb E\,;\, B_r(\bar{x}))
    +C r^{N-1+\left(1-\frac{N(p-1)}{q}\right)}\,,
\end{equation}
for any set $E\subset B_r(\bar{x})$ of finite perimeter, with $E\triangle (\o_i\cap B_r(\bar{x}))\subset\subset B_r(\bar{x})$.
If \eqref{eq:minloc} holds, then using the regularity results by Tamanini (see Theorem \ref{thm:regTam}) we obtain that $\rb \o_i$ is, up to a closed $\hno$-negligible set, a relatively open hypersurface of class $C^{1,\frac{1}{2p}}$.

We now prove \eqref{eq:minloc}. Let $E\subset B_r(\bar{x})$ be a set of finite perimeter with $E\triangle (\o_i\cap B_r(\bar{x}))\subset\subset B_r(\bar{x})$. Define the function
\begin{equation}\label{eq:defv}
v(x):=
\left\{
\begin{array}{ll}
a_i & \text{ if } x\in V_i\,,\\
a_j & \text{ if } x\in V_j\,,\\
u(x) & \text{ otherwise}\,,
\end{array}
\right.
\end{equation}
where $V_i:=E\cap B_r(\bar{x})$ and $V_j:=B_r(\bar{x})\setminus E$, and we recall that $u\in BV(\o;A)$, where $A:=\{a_1,\dots, a_k\}\subset\R^M$.
The minimality of $u$ yields
\[
\f(u)\leq \f(v)\,,
\]
from which we get
\begin{equation}\label{eq:prevestimate}
\alpha_{ij} \hno(\rb\o_i\,;\, B_r(x)) \leq \alpha_{ij}\hno(\rb E\,;\, B_r(x))+R\,,
\end{equation}
where
\begin{align}\label{eq:r}
R&:=\lambda \int_{B_r(\bar{x})\,\setminus\,D}  \left( |u-f|^p-|v-f|^p \right) \dx \nonumber\\
&\hspace{1cm}+\mu\int_{B_r(\bar{x})\cap D} \left( |L(u\cdot e)-L(f\cdot e)|^p-|L(v\cdot e)-L(f\cdot e)|^p \right) \dx\,. 
\end{align}
We want to estimate $R$. For the sake of simplicity, in what follows $C>0$ will denote a constant that might change from line to line. Using the definition of $v$ (see \eqref{eq:defv}) and arguing as in \eqref{eq:estid}, we get
\begin{align}\label{eq:estimateoutD}
\Bigg|\,&\int_{B_r(\bar{x})\,\setminus\,D}  \left( |u-f|^p-|v-f|^p \right) \dx\,\Bigg|  \\
&\hspace{0.6cm}\leq C\int_{[(\o_i\triangle V_i)\cup(\o_j\triangle V_j)]\setminus D}
        \left[\, |u-f|^{p-1}+|v-f|^{p-1} \,\right]\dx \leq C r^{N\left( 1-\frac{p-1}{q} \right)} \nonumber\,,
\end{align}
where in the last step we used the fact that $\o_i\triangle V_i\subset\subset B_r(\bar{x})$ and $\o_j\triangle V_j\subset\subset B_r(\bar{x})$.
A similar argument yields
\begin{align}\label{eq:estimateD}
\Bigg|\,\int_{B_r(\bar{x})\cap D} & \left( |L(u\cdot e)-L(f\cdot e)|^p
    -|L(v\cdot e)-L(f\cdot e)|^p \right) \dx\,\Bigg|
\leq C r^{N\left( 1-\frac{p-1}{q} \right)}\,.
\end{align}
Using \eqref{eq:prevestimate}, \eqref{eq:estimateoutD}, \eqref{eq:estimateD}, and the fact that $\min_{i,j}\alpha_{ij}>0$,
we deduce that
\[
\hno(\rb\o_i\,;\, B_r(x))\leq \hno(\rb E\,;\, B_r(x))
    +C r^{N\left( 1-\frac{p-1}{q} \right)}\,,
\]
and this proves \eqref{eq:minloc}.\\

\emph{Step 3: Proof of $(iii)$.}
Let
\begin{equation}\label{eq:beta}
\beta:=\min\left\{\, \gamma_1, \frac{\eta^{\frac{N-1}{N}}}{\gamma_2}  \,\right\}\,,
\end{equation}
where $\gamma_1, \gamma_2>0$ are the constants given by Lemma \ref{lem:genisop} and $\eta>0$ is the one given by Theorem \ref{thm:elimination}.
Let $x\in\overline{J_u}\cap\o$, and assume that
\begin{equation}\label{eq:absurd}
\liminf_{r\to0}\frac{\hno(J_u\cap B_r(x))}{r^{N-1}}<\beta\,.
\end{equation}
Find $r\in(0,r_0)$, where $r_0>0$ is given by Theorem \ref{thm:elimination}, such that
\[
\frac{\hno(J_u\cap B_r(x))}{r^{N-1}}<\gamma_1\,.
\]
By Lemma \ref{lem:genisop} there exists an index $\bar{j}\in\{1,\dots,k\}$ such that
\begin{equation}\label{eq:estlemgen}
|B_r(x)\setminus \o_{\bar{j}}|^{\frac{N-1}{N}}\leq \gamma_2 \hno(J_u\cap B_r(x))\,.
\end{equation}
Using \eqref{eq:beta}, \eqref{eq:absurd} and \eqref{eq:estlemgen}, we get
\[
|B_r(x)\setminus \o_{\bar{j}}|\leq \gamma_2^{\frac{N}{N-1}}\left(\hno(J_u\cap B_r(x))\right)^{\frac{N}{N-1}}
    \leq (\gamma_2\beta)^{\frac{N}{N-1}}r^N\leq \eta r^N\,.
\]
Applying Theorem \ref{thm:elimination} and using Remark \ref{rem:perm}, we get that
$|V\cap B_{r/2}(x)|=0$ for every
\[
V:=\bigcup_{j\in \{1,\dots,k\}\setminus\{\bar{j}\}} \o_{\sigma(j)}\,,
\]
where $\sigma:\{1,\dots,k\}\setminus\{\bar{j}\}\to\{1,\dots,k\}\setminus\{\bar{j}\}$ is a permutation. Thus, we obtain that
\[
|B_{r/2}(x)\cap \o_j|=0
\]
for all $j\in\{1,\dots,k\}\setminus\{\bar{j}\}$.
Arguing as we did in Step 1, we obtain that $B_{r/2}(x)\subset \o_{\bar{j}}$.
Thus, $x\not\in\overline{J_u}\cap\o$. This contradicts our initial assumption.
In particular, we conclude $(iv)$.
\end{proof}



We are now in position to prove the main result of this paper.

\begin{proof}[Proof of Theorem \ref{thm:kcolors}]
Notice that the minimization problem \eqref{eq:kcolors} can be written as
\[
\inf_{a_1,\dots,a_k\in\R^M}\,\inf_{u\in BV(\o;\{a_1,\dots,a_k\})}\f(u)\,.
\]
Let $\{(a_1^n,\dots,a_k^n)\}_{n\in\N}$ and $\{u^n\}_{n\in\N}$ be minimizing sequences for the minimization problem \eqref{eq:kcolors}, \emph{i.e.}, $u^n\in BV(\o;\{a_1^n,\dots,a_k^n\})$ and
\[
\lim_{n\to\infty}\f(u^n)=\inf_{a_1,\dots,a_k\in\R^M}\,\inf_{u\in BV(\o;\{a_1,\dots,a_k\})}\f(u)\,.
\]
Without loss of generality, we can assume that
\begin{equation}\label{eq:energybound}
\lim_{n\to\infty}\f(u^n)<\infty\,.\\
\end{equation}

\emph{Step 1.} We claim that there exists $\bar{j}\in\{1,\dots,k\}$ such that
\[
\sup_{n\in\N}|a_{\bar{j}}^n|<\infty\,.
\]
Indeed, write
\[
u^n=\sum_{i=1}^k a^n_i\chi_{\o^n_i}\,.
\]
Since for every $n\in\N$ it holds that $|\o\setminus\cup_{i=1}^k \o^n_i|=0$, it is possible to find $\bar{j}\in\{1,\dots,k\}$ and $\delta>0$ such that (up to a not relabeled subsequence),
\begin{equation}\label{eq:boundbelow}
|\o^n_{\bar{j}}\setminus D|\geq\delta\,,
\end{equation}
for all $n\in\N$. We have
\begin{align*}
|a^n_{\bar{j}}|&\leq \frac{1}{|\o^n_{\bar{j}}\setminus D|^{\frac{1}{p}}}
   \left[\,\| a^n_{\bar{j}}-f \|_{L^p(\o^n_{\bar{j}}\setminus D;\R^M)} + \|f \|_{L^p(\o^n_{\bar{j}}\setminus D\,;\,\R^M)} \,\right]\\
&\leq \frac{1}{\delta^{\frac{1}{p}}} \left[\, \frac{1}{\lambda}\f(u^n) + \|f\|_{L^p(\o\,;\,\R^M)}   \,\right]\,.
\end{align*}
We conclude using \eqref{eq:energybound} and the fact that $f\in L^p(\o\,;\,\R^M)$.\\

\emph{Step 2.} We claim that, up to a (not relabeled) subsequence, for every $i\in\{1,\dots,k\}$ the following holds:
either $\{a^n_i\}_{n\in\N}$ is bounded or $|\o^n_i|\to0$ as $n\to\infty$.

Indeed, consider the sequence of sets $\{\o^n_{\bar{j}}\}_{n\in\N}$, where $\bar{j}\in\{1,\dots,k\}$ is an index found in Step 1. We have two cases: either there exists a (not relabeled) subsequence for which
\begin{equation}\label{eq:boundbelowper}
\hno(\rb\o^n_{\bar{j}}\cap\o)\geq\widetilde{\delta}\,,
\end{equation}
for every $n\in\N$ and for some $\widetilde{\delta}>0$, or $\hno(\rb\o^n_{\bar{j}}\cap\o)\to0$ as $n\to\infty$.
In the latter case, from \eqref{eq:boundbelow} and the isoperimetric inequality in $\o$ (see \cite[Remark 12.38]{Maggi}), we get that $\chi_{\o^n_{\bar{j}}}\to\chi_{\o}$ in $L^1(\o)$, and the claim is proved.
Assume that \eqref{eq:boundbelowper} holds. 
Since
\[
\hno(\rb\o^n_{\bar{j}}\cap\o)=\sum_{i\neq \bar{j}} \hno(\rb\o^n_{\bar{j}}\cap\rb\o^n_i\cap \o)\,,
\]
we can find, up to a (not relabeled) subsequence, an index $i\in\{1,\dots,k\}\setminus\{\bar{j}\}$ such that
\[
\inf_{n\in\N}\hno(\rb\o^n_{\bar{j}}\cap\rb\o^n_i\cap \o)>0\,.
\]
Using \eqref{eq:energybound} and the fact that $\{a^n_{\bar{j}}\}_{n\in\N}$ is bounded (see Step 1), we deduce that also $\{a^n_i\}_{n\in\N}$ is bounded.

We then proceed by induction as follows: assume that we found indexes $j_1,\dots,j_s\in\{1,\dots,k\}$, for some $s\in\{1,\dots,k\}$, such that, for all $i=1,\dots,s$, $\{a^n_{j_i}\}_{n\in\N}$ is bounded.
Consider the sequence of sets $\{V_n\}_{n\in\N}$, where
\[
V_n:=\bigcup_{i=1}^s \o^n_{j_i}\,.
\]
Then, either there exists a (not relabeled) subsequence for which
\[
\hno(\rb V_n\cap\o)\geq\widetilde{\delta}
\]
for every $n\in\N$ and for some $\widetilde{\delta}>0$, or $\hno(\rb V_n\cap\o)\to0$ as $n\to\infty$.
Reasoning as above, in the former case we get that $|\o^n_i|\to0$ as $n\to\infty$ for all $i\in\{1,\dots,k\}\setminus\{j_1,\dots,j_s\}$, while in the latter case we find an index $i\in\{1,\dots,k\}\setminus\{j_1,\dots,j_s\}$ such that $\{a^n_i\}_{n\in\N}$ is bounded.
Since $k\in\N$ is finite, this proves the claim.\\

\emph{Step 3.} We now conclude as follows. Denote by $I\subset\{1,\dots,k\}$ the set of indexes $i\in \{1,\dots,k\}$ for which the sequence $\{a^n_i\}_{n\in\N}$ is bounded.
Using a diagonalizing argument, and up to a subsequence (not relabeled), we have that
\begin{equation}\label{eq:convain}
a^n_i\to a_i\,,
\end{equation}
as $n\to\infty$, for all $i\in I$. Set $a_i:=0$ for all $i\in\{1,\dots,k\}\setminus I$.

\emph{Case 1.} Assume that $a_i\neq a_j$ if $i, j\in I$ with $i\neq j$.
In this case, we have that $\inf_{n\in\N}|a^n_i-a^n_j|>0$, for all $i, j\in I$ with $i\neq j$.
Using \eqref{eq:energybound} we obtain
\[
\sup_{n\in\N}\hno\left(\, \rb\o^n_i\cap \o \,\right)<\infty\,.
\]
Hence, for all $i\in I$, due to the compactness for sets of finite perimeter (see \cite[Theorem 3.23]{AFP}) up to a subsequence (not relabeled),
\[
\chi_{\o^n_i}\to\chi_{\o_i}\quad\quad \text{ in } L^1(\o)
\]
for some set of finite perimeter $\o_i\subset\o$.
On the other hand, we know that $|\o^n_i|\to0$ as $n\to\infty$ for all $i\in\{1,\dots,k\}\setminus I$.
Define the function $u\in BV(\o\,;\,\{a_1,\dots,a_k\})$ as
\[
u:=\sum_{i=1}^k a_i\chi_{\o_i}\,,
\]
where we set $\o_i:=\emptyset$ for all $i\in \{1,\dots,k\}\setminus I$. We claim that $u$ is a solution to the 
minimization problem \eqref{eq:kcolors}.
Indeed, setting $\alpha^n_{ij}:=|a_i^n-a_j^n|$ for $i,j\in I$ and $n\in\N$, we get
\begin{align*}
\liminf_{n\to\infty}\f(u^n)&\geq \liminf_{n\to\infty} \,\Biggl[\, \sum_{i<j \in I} \alpha^n_{ij}\hno(\rb \o^n_i\cap\rb \o^n_j\cap\o)\\
&\hspace{2cm}+ \lambda\sum_{i\in I} \int_{\o^n_i\,\meno\, D} |a^n_i-f|^p\dx  \\
&\hspace{3cm}+ \mu\sum_{i\in I} \int_{\o^n_i\cap D} |L(a^n_i\cdot e)-L(f\cdot e)|^p\dx \,\Biggr] \\
&\geq\sum_{i<j\in I} \alpha_{ij}\hno(\rb \o_i\cap\rb \o_j\cap\o)
    + \lambda\sum_{i\in I} \int_{\o_i\,\meno\, D} |a_i-f|^p\dx \\
&\hspace{2cm}+ \mu\sum_{i\in I} \int_{\o_i\cap D} |L(a_i\cdot e)-L(f\cdot e)|^p\dx\\
&=\f(u)\,,
\end{align*}
where in the last inequality we used the lower semicontinuity result by Ambrosio and Braides (see \cite[Theorem 2.1]{AmbBra2}), together with the facts that
\[
\int_{\o^n_i\,\meno\, D} |a^n_i-f|^p\dx\to \int_{\o_i\,\meno\, D} |a_i-f|^p\dx\,,
\]
and
\[
\int_{\o^n_i\cap D} |L(a^n_i\cdot e)-L(f\cdot e)|^p\dx\to \int_{\o_i\cap D} |L(a_i\cdot e)-L(f\cdot e)|^p\dx\,,
\]
as $n\to\infty$, for every $i\in I$. This proves the existence of a solution to the 
minimization problem \eqref{eq:kcolors} in this case.\\

\emph{Case 2.} Assume that $a_i=a_j$ for some $i, j\in I$ with $i\neq j$. We reason as follows.
Without loss of generality, we can suppose that there exist $s\in\{1,\dots,k\}$ and $b_1,\dots,b_s\in\{1,\dots,k\}$ with 
\[
1=b_1<b_2<b_3<\dots<b_s=k
\]
such that $a_i=a_j$ for all $i\in\{ b_j,\dots,b_{j-1}-1 \}$.
Consider the sequence $\{(V^n_1,\dots,V^n_k)\}_{n\in\N}$ of $k$-finite Caccioppoli partitions defined as
\[
V^n_i:=\o^n_{b_i}\cup\o^n_{b_i+1}\dots\cup \o^n_{b_{i+1}-1}\,,
\]
if $i\in\{1,\dots,s\}$, and $V^n_i:=\emptyset$ for $i\in\{s,\dots,k\}$,
and the sequence of  functions $\{v^n\}_{n\in\N}$ given by
\[
v^n:=\sum_{i=1}^k b^n_i \chi_{V^n_i}\,,
\]
where $b^n_i:=a^n_i$, for all $n\in\N$ and all $i=1,\dots,s$, while we set $b^n_i:=0$ for all $i=s,\dots,k$.
Applying the reasoning of Case 1 to the sequences $\{b^n_1,\dots, b^n_k\}$ and $\{v^n\}_{n\in\N}$, we get the existence of a solution of the minimization problem \eqref{eq:kcolors}.\\

\emph{Step 4.} Let $u\in BV(\o;\{a_1,\dots,a_k\})$ be a solution to the minimization problem \eqref{eq:kcolors}.
In particular, $u$ is a solution to the minimization problem
\[
\min_{v\in BV(\o;\{a_1,\dots,a_k\})} \f(v)\,.
\]
Thus, under the additional hypotheses $(H1), (H2)$ and $(H3)$, Theorem \ref{thm:kfixedcolors} yields the regularity result.
\end{proof}

\begin{remark}
We remark that, as it is well known in the literature, if in the minimization problem \eqref{eq:kcolors} we allow countably many colors, then, in general, the problem does not admit a solution.
Indeed, let $f\in L^\infty(\o;\R^M)$ be an initial datum for which any solution of the minimization problem \eqref{eq:minpb} is not piecewise constant. Let $u$ denote one of these solutions.
If there was a solution $v$ to the minimization problem \eqref{eq:kcolors} when countably many colors are allowed, we would have $\f(u)<\f(v)$. Since the class of piecewise constant functions is dense in the set of bounded functions with bounded variation with respect to the $L^p$ topology, for any $p\geq1$, it is possible to find a sequence $\{u_n\}_{n\in\N}$ of bounded piecewise constant functions converging in $L^1$ to $u$, with $|Du_n|(\o)\to|Du|(\o)$ as $n\to\infty$. In particular, $\f(u_n)\to\f(u)$. Thus, for $n$ large we would have $\f(u_n)<\f(v)$, contradicting the minimality of $v$.
\end{remark}

\section*{Acknowledgement}
The authors thank the Center for Nonlinear Analysis at Carnegie Mellon University for its support
during the preparation of the manuscript.
The research of both authors was funded by National Science Foundation under Grant No. DMS-1411646.


\bibliographystyle{siam}
\bibliography{Biblio}


\end{document}